\documentclass[12pt]{article}
\usepackage{a4}
\usepackage{amsmath}
\usepackage{amssymb}
\usepackage{amsfonts}
\usepackage{amsthm}
\usepackage{graphicx}
\usepackage{color}
  \newtheorem{thm}{Theorem}[section]
 \newtheorem{cor}[thm]{Corollary}

 \newtheorem{lemma}[thm]{Lemma}

\def\R{\mathbb{R}}

\def\Z{\mathbb{Z}}

\author {Christian Berg and Henrik L.\ Pedersen}
\title {On the entropy of the log-normal distribution}

\date{\today}

\begin{document}
\maketitle

\begin{abstract} In a paper by Novi Inverardi and Tagliani from 2024  it is claimed that the log-normal density has maximal entropy in the class of densities having the same moments as the log-normal density. We disprove the claim by finding a family of densities with larger entropy.
 \end{abstract}


\section{Introduction and main results}
For a probability density $f$  on an interval $I$ the quantity
\begin{equation}\label{eq:ent}
H[f]:=-\int_I f(x)\log f(x)\,dx   
\end{equation}
is called the differential entropy or just the entropy of $f$, cf.\ \cite {C:T}, \cite{K:K}.

A Stieltjes moment sequence is a sequence of numbers $(m_k)_{k\ge 0}$ for which there exists a positive measure $\mu$ on the interval $[0,\infty)$  with moments of any order satisfying
\begin{equation}\label{eq:mom}
m_k=\int_0^\infty x^k\,d\mu(x),\quad k=0,1,\ldots.
\end{equation} 
Since we will be dealing only with probability measures $\mu$, we assume that the moment sequence starts with $m_0=1$. The moment sequence is called Stieltjes determinate if there is only one probability measure on $[0,\infty)$ satisfying \eqref{eq:mom}, and it is  called Stieltjes indeterminate, if there are more than one such measure, and in this case the set of all such measures is an infinite convex set $W$ which is compact in  the weak topology, cf. \cite{Sch}.   

In the Stieltjes indeterminate case the subset $\mathcal C\subset W$ consisting of probability densities is dense in $W$, but $W$ also contains many discrete measures as well as continuous singular measures. The first examples of Stieltjes indeterminate measures were given by Stieltjes \cite{St}, and among those is the log-normal density

\begin{equation}\label{eq:ln}
g(x)=\frac{1}{\sqrt{2\pi}}x^{-1}\exp(-\log^2(x)/2),\quad x>0,
\end{equation}
with moments
$$
m_k=\exp(k^2/2),\quad  k=0,1,\ldots.
$$ 
Stieltjes' proof consists in showing that the family
of densities
\begin{equation}\label{eq:SC}
g_r(x)=g(x)[1+r\sin(2\pi\log x)],\quad -1\le r\le 1,
\end{equation}
has the same moments $m_k$, which is true because

$$
t_k:=\int_0^\infty x^{k} g(x)\sin(2\pi\log x)\,dx=0,\quad k\in\Z.
$$ 
 To see this we apply the substitution $t=\log x$ and get
\begin{eqnarray*}
t_k&=&\frac{1}{\sqrt{2\pi}}\int_{-\infty}^\infty \exp(tk-t^2/2)\sin(2\pi t)\,dt\\
&=& \frac{\exp(k^2/2)}{\sqrt{2\pi}}\int_{-\infty}^\infty \exp(-(t-k)^2/2)\sin(2\pi t)\,dt \\
&=& \frac{\exp(k^2/2)}{\sqrt{2\pi}}\int_{-\infty}^\infty \exp(-t^2/2)\sin(2\pi(t+k))\,dt\\
&=& \frac{\exp(k^2/2)}{\sqrt{2\pi}}\int_{-\infty}^\infty \exp(-t^2/2)\sin(2\pi t)\,dt=0,
\end{eqnarray*}
where we used the periodicity of $\sin$ and that it is a odd function.

It is easy to evaluate the entropy $H[g]$ of the log-normal density $g$. We get
\begin{equation}\label{eq:entlogn}
H[g]=(1+\log(2\pi))/2.
\end{equation}

A family of densities like \eqref{eq:SC} with the same moments  is called a Stieltjes class in \cite{Sto}. A variant of the log-normal distribution and many of the distributions having the same moments are studied in \cite{Chr}. In the paper \cite{B} there is a study of the entropy for a class of densities which are solutions to  the general indeterminate Hamburger moment problem.

Various papers have been exploiting the existence of a density $f_{hmax}\in\mathcal C$ such that
\begin{equation}\label{eq:max}
H[f_{hmax}]=\max_{f\in\mathcal C}H[f],
\end{equation}
cf. \cite{N:T2}. 

In \cite[Theorem 2]{N:T3} the authors claim that  the log-normal density $g$ in \eqref{eq:ln} is equal to the density with maximal entropy. 

The goal of the present paper is to demonstrate that this is not true because of the following result:

\begin{thm}\label{thm:1} The entropy $H[g_r]$ of the family of densities \eqref{eq:SC} is given as the power series
\begin{equation}\label{eq:pow}
H[g_r]=H[g]+\alpha r -\sum_{k=1}^\infty \frac{\beta_k}{2k(2k-1)}r^{2k},\quad -1\le r\le 1,
\end{equation}
where
\begin{eqnarray}
\alpha&=&\frac{1}{\sqrt{2\pi}}\int_{-\infty}^\infty \exp(-t^2/2) t\sin(2\pi t)\,dt,\label{eq:alpha}\\
\beta_k&=&\frac{1}{\sqrt{2\pi}}\int_{-\infty}^\infty\exp(-t^2/2)\sin^{2k}(2\pi t)\,dt, k=1,2,\ldots.\label{eq:betak}
\end{eqnarray}
\end{thm}

\begin{cor}\label{thm:2} For the family of densities $g_r$ from \eqref{eq:SC} we have 
\begin{equation}
H[g]<\max_{r\in[-1,1]}H[g_r]
\end{equation}
and the maximum is realized for a unique number $0<r_0<\alpha/\beta_1$.
\end{cor}

Exact values of $\alpha$ and $\beta_k$ together with numerical approximations of $\alpha$, $r_0$ and $H[g_{r_0}]$ are given in the next lemmas. The numerical results illustrate that $H[g_{r_0}]$ is very close to $H[g]$.
\begin{lemma}\label{thm:evalf} We have
$$
\alpha=2\pi\exp(-2\pi^2),\quad \beta_1=(1-\exp(-8\pi^2))/2
$$
and in general, for $k\geq 1$,
$$
\beta_k=4^{-k}\left(\binom{2k}{k}+2\sum_{\ell=1}^{k}\binom{2k}{k+\ell}(-1)^{\ell}\exp(-8\pi^2\ell^2)\right).
$$
\end{lemma}

\begin{lemma} For the reader's convenience we state the following numerical approximations.
\begin{align*}
H[g]&\approx 1.4189\\
H[g_{-1}]&\approx 1.1121\\
H[g_{1}]&\approx 1.1121\\
\alpha&\approx 1.68\cdot 10^{-8}\\
H[g_{1}]-H[g_{-1}]=2\alpha&\approx 3.36\cdot 10^{-8}\\
 0<r_0<\alpha/\beta_1&\approx 3.36\cdot 10^{-8}\\
 r_0&\approx 3.36\cdot 10^{-8}\\
0<H[g_{r_0}]-H[g]<\alpha^2/\beta_1&\approx 5.65\cdot 10^{-16}\\
H[g_{r_0}]-H[g]&\approx 2.83\cdot 10^{-16}
\end{align*}
\end{lemma}
Numerical approximation of $r_0$ using \eqref{eq:r0} below, where the series is truncated  to $1\leq k\leq 150$ and numerical approximation of the upper bound $\alpha/\beta_1$ indicate that the two values are equal up the order of $10^{-23}$. The computation of the difference $H[g_{r_0}]-H[g]$ based on a similar truncation of the series in \eqref{eq:pow} is slightly smaller than using the upper bound $\alpha^2/\beta_1$. Approximations for $r=-1$ and for $r=1$ have also been calculated using truncation of \eqref{eq:pow}. The computations have been carried out using Maple software.

\section{Proofs}
Let us first draw a consequence of Theorem~\ref{thm:1} and Lemma~\ref{thm:evalf} by noticing that the constants $\alpha,\beta_k$ are all positive. Since $(\beta_k)_{k\geq 1}$ is a decreasing sequence tending to zero by Lebesgue's Theorem on dominated convergence, the expression \eqref{eq:pow} shows that $H[g_r]$ is continuous on the interval $[-1,1]$ and $C^\infty$ on $(-1,1)$ with a strictly negative second derivative, hence a strictly concave function. The derivative on $(-1,1)$
$$
H[g_r]'=\alpha-\sum_{k=1}^\infty\frac{\beta_k}{2k-1}r^{2k-1}
$$
has the value $\alpha>0$ for $r=0$, and hence $H[g_r]>H[g]$ on a sufficiently small interval to the right of 0. This already shows that  the statement of \cite[Theorem 2]{N:T3} is wrong. Furthermore, the maximum of $H[g_r]$ occurs at the unique solution $r_0>0$ of the equation $H[g_r]'=0$, i.e., given by
\begin{equation}\label{eq:r0}  
\alpha=\sum_{k=1}^\infty\frac{\beta_k}{2k-1}r_0^{2k-1},
\end{equation}
and from this equation we get $\alpha>\beta_1 r_0$, hence
 $0<r_0<\alpha/\beta_1$. This gives the upper bound $H[g_{r_0}]<H[g]+\alpha^2/\beta_1$.

This shows that Corollary~\ref{thm:2} follows from Theorem~\ref{thm:1}. 

\medskip
{\bf Proof of Theorem~\ref{thm:1}:}
From \eqref{eq:ent} we get
\begin{align*}
H[g_r]&=-\int_0^\infty g(x)[1+r\sin(2\pi\log x)][\log g(x)+\log(1+r\sin(2\pi\log x))]\,dx\\
&=H[g] -r\int_0^\infty g(x)\log g(x) \sin(2\pi\log x)\,dx\\
&\phantom{=}-\int_0^\infty g(x)\log(1+r\sin(2\pi\log x))\,dx \\
&\phantom{=} -r\int_0^\infty g(x) \sin(2\pi\log x)\log(1+r\sin(2\pi\log x))\,dx,
\end{align*}
which is an expression with four terms. Denoting the last three terms $T_1,T_2,T_3$, we get using the substitution $t=\log x$
\begin{align*}
T_1&=-r\frac{1}{\sqrt{2\pi}}\int_{-\infty}^\infty \exp(-t^2/2)\sin(2\pi t)[-\log\sqrt{2\pi}-t-t^2/2]dt\\
&=r\frac{1}{\sqrt{2\pi}}\int_{-\infty}^\infty \exp(-t^2/2) t\sin(2\pi t)\,dt=r\alpha,
\end{align*}
where $\alpha$ is given in \eqref{eq:alpha}.
Using the  power series expansion of $\log(1+z)$ we find for $|r|<1$
\begin{align*}
T_2&=\sum_{n=1}^\infty (-1)^n\frac{r^{n}}{n} \frac{1}{\sqrt{2\pi}}\int_{-\infty}^\infty \exp(-t^2/2)\sin^{n}(2\pi t)\,dt\\
&= \sum_{k=1}^\infty \frac{r^{2k}}{2k}
\frac{1}{\sqrt{2\pi}}\int_{-\infty}^\infty\exp(-t^2/2)\sin^{2k}(2\pi t)\,dt,
\end{align*}
and
\begin{align*}
T_3&=\sum_{n=1}^\infty (-1)^n\frac{r^{n+1}}{n} \frac{1}{\sqrt{2\pi}}\int_{-\infty}^\infty \exp(-t^2/2)\sin^{n+1}(2\pi t)\,dt\\
&=  -\sum_{k=1}^\infty \frac{r^{2k}}{2k-1}\frac{1}{\sqrt{2\pi}}\int_{-\infty}^\infty\exp(-t^2/2)\sin^{2k}(2\pi t)\,dt.
 \end{align*}
This gives
$$
T_2+T_3=-\sum_{k=1}^\infty\frac{r^{2k}}{2k(2k-1)}\beta_k,
$$
with $\beta_k$ given in \eqref{eq:betak}.

Summing up we have
$$
H[g_r]=H[g]+\alpha r-\sum_{k=1}^\infty \frac{r^{2k}}{2k(2k-1)}\beta_k,
$$
which is \eqref{eq:pow}. $\square$

\medskip
{\bf Proof of Lemma~\ref{thm:evalf}:}
The Fourier transform of the normal density is given as
\begin{equation}\label{eq:Fou}
\frac{1}{\sqrt{2\pi}}\int_{-\infty}^\infty\exp(-t^2/2)\exp(itb)\,dt=\exp(-b^2/2),\quad b\in\R,
\end{equation}
hence
$$
\frac{1}{\sqrt{2\pi}}\int_{-\infty}^\infty\exp(-t^2/2)\cos(tb)\,dt=\exp(-b^2/2).
$$
Differentiation of this equation with respect to $b$  yields
$$
\frac{1}{\sqrt{2\pi}}\int_{-\infty}^\infty\exp(-t^2/2)t\sin(tb)\,dt=b\exp(-b^2/2),
$$
and this gives, for $b=2\pi$,
$\alpha=2\pi\exp(-2\pi^2)$.

Using Euler's formula we obtain
\begin{align*}
 \sin^{2k}(bt)&=\frac{(-1)^k}{4^k}\sum_{\ell=0}^{2k}\binom{2k}{\ell}\exp(ibt\ell)(-1)^{2k-\ell}\exp(-ibt(2k-\ell))\\
 &=\frac{(-1)^k}{4^k}\sum_{\ell=0}^{2k}\binom{2k}{\ell}(-1)^{\ell}\exp(2ibt(\ell-k))\\
 &=4^{-k}\sum_{\ell=-k}^{k}\binom{2k}{k+\ell}(-1)^{\ell}\exp(2ibt\ell).
\end{align*}
This gives
\begin{align*}
 \beta_k&=4^{-k}\sum_{\ell=-k}^{k}\binom{2k}{k+\ell}(-1)^{\ell}\frac{1}{\sqrt{2\pi}}\int_{-\infty}^\infty\exp(-t^2/2)\cos(4\pi t\ell)\,dt\\
 &=4^{-k}\sum_{\ell=-k}^{k}\binom{2k}{k+\ell}(-1)^{\ell}\exp(-8\pi^2\ell^2)\\
 &=4^{-k}\left(\binom{2k}{k}+2\sum_{\ell=1}^{k}\binom{2k}{k+\ell}(-1)^{\ell}\exp(-8\pi^2\ell^2)\right).
\end{align*}
$\square$

\noindent
Christian Berg\\
Department of Mathematical Sciences\\
University of Copenhagen \\
Universitetsparken 5\\
DK-2100, Denmark\\
{\em email}:\hspace{2mm}{\tt berg@math.ku.dk}

\vspace{0.5cm}

\noindent
Henrik Laurberg Pedersen\\
Department of Mathematical Sciences\\
University of Copenhagen \\
Universitetsparken 5\\
DK-2100, Denmark\\
{\em email}:\hspace{2mm}{\tt henrikp@math.ku.dk}

\end{document}